\documentclass{birkjour}
 \newtheorem{thm}{Theorem}
 
 \newtheorem{lem}[thm]{Lemma}
 
 \theoremstyle{definition}
 \newtheorem{defn}[thm]{Definition}
 \newtheorem*{theo}{Theorem}
 \theoremstyle{remark}

 \usepackage[mathscr]{eucal}
\usepackage[english]{babel}
\newcommand{\ccomma}{\mathpunct{\raisebox{0.5ex}{,}}}
\DeclareMathOperator{\re}{Re}

\begin{document}

%
%
%
%
%

\title[On measures which generate the scalar product]{On measures which generate the scalar\\ product in a space of rational functions.}
\author[V. Katsnelson]{Victor Katsnelson}
\address{%
Department of Mathematics\\
The Weizmann Institute\\
76100, Rehovot\\
Israel}
\email{victor.katsnelson@weizmann.ac.il; victorkatsnelson@gmail.com}
\subjclass{46E20}
\keywords{Non-negative measures on the unite circle; rational fractions
with poles in the unit disc.}
\date{February, 2016}
\begin{abstract}
Let \(z_1,z_2,\,\ldots\,,z_n\) be pairwise different points of the unit
disc and \(\mathscr{L}(z_1,z_2,\,\ldots\,z_n)\) be the linear space generated by the rational fractions
\(\frac{1}{t-z_1}\ccomma\frac{1}{t-z_2}\ccomma\,\cdots\,\ccomma
\frac{1}{t-z_n}\cdot\) Every non-negative measure \(\sigma\) on the
unit circle \(\mathbb{T}\) generates the scalar product
\[\langle\,f\,,\,g\,\rangle_{\!_{L^2_\sigma}}
=\int\limits_{\mathbb{T}}f(t)\,\overline{g(t)}\,\sigma(dt),
\quad \forall\,f,g\,\in\,L^2_\sigma.\]
The measures \(\sigma\) are described which satisfy the condition
\[\langle\,f\,,\,g\,\rangle_{\!_{L^2_\sigma}}=
\langle\,f\,,\,g\,\rangle_{\!_{L^2_m}},\quad \forall\,f,g\in\mathscr{L}(z_1,z_2,\,\ldots\,z_n),\]
where \(m\) is the normalized Lebesgue measure on \(\mathbb{T}\).
\end{abstract}
\maketitle
\begin{itemize}
\item  \(\mathbb{R}\) stands for the set of all real numbers.
\item  \(\mathbb{C}\) stands for the set of all complex numbers.
\item  \(\mathbb{T}\) is the unite circle:
\(\mathbb{T}= \{t\in\mathbb{C}:\,|t|=1\}.\)
\item \(\mathbb{D}_+\) is the open unite disc: \(\mathbb{D}_+= \{z\in\mathbb{C}:\,|z|<1\}.\)
\item \(\mathbb{D}_-\) is the exterior of the unite circle: \(\mathbb{D}_-= \{z\in\mathbb{C}:\,1<|z|\leq \infty\}.\)
\end{itemize}

 \vspace{3.0ex}
The main objects of this note are spaces of functions on the unit circle \(\mathbb{T}\) and measures generated scalar products in such spaces.
Let us set some notation.

 For
 a non-negative measure \(\sigma\) on \(\mathbb{T}\),
the space \(L^2_\sigma\) is the set of all square integrable functions with respect to the measure \(\sigma\). The space \(L^2_\sigma\) is provided by natural linear operations and by the scalar product
\begin{equation*}
\langle\,f\,,\,g\,\rangle_{\!_{L^2_\sigma}}=\int\limits_{\mathbb{T}}f(t)\,\overline{g(t)}\,\sigma(dt).
\quad \forall\,f,g\,\in\,L^2_\sigma.
\end{equation*}
Let \(z_1,z_2,\,\ldots\,,\,z_n,\,n<\infty,\) be a sequence of pairwise different points of the unit disc \(\mathbb{D_+}\):
\begin{subequations}
\label{seq}
\begin{gather}
z_k\in\mathbb{D_+},\quad 1\leq k \leq n,\label{seq_a} \\
z_p\not=z_q, \quad 1\leq p,q\leq n.\label{seq_b}
\end{gather}
\end{subequations}
We relate  two objects to this sequence. The first object is the Blaschke product%
\footnote{If \(z_k=0\) for some \(k\), then the appropriate factor of the Blaschke product \eqref{BP} is equal just to \(t\), but not to
\(\frac{z_k-t}{1-\overline{z_k}t}\cdot\frac{|z_k|}{z_k}\cdot\).}
\begin{equation}
\label{BP}
B(t)=\prod\limits_{1\leq k\leq n}\frac{z_k-t}{1-\overline{z_k}t}\cdot\frac{|z_k|}{z_k}\cdot
\end{equation}
The second object is the sequence
\begin{equation}
\label{RF}
 e(t,z_k)=\frac{1}{t-z_k},\quad 1\leq k\leq n,
 \end{equation}
 of rational fractions.

 Let \(\mathscr{L}(z_1,z_2,\,\ldots\,,z_n)\) be the linear space generated
 by the functions \( e(t,z_k),\,1\leq k\leq n\). In other words,
 \(\mathscr{L}(z_1,z_2,\,\ldots\,,z_n)\) is the set of all linear combinations
 \begin{math}
 \sum\limits_{1\leq k\leq h}\alpha_k\,  e(t,z_k),
 \end{math}
 where \(\alpha_k,\,1\leq k\leq n,\) are arbitrary complex numbers.

 Functions which belong to the linear space \(\mathscr{L}(z_1,z_2,\,\ldots\,,z_n)\) are continuous on the unit circle
 \(\mathbb{T}\). Therefore they belong to every space \(L^2_\sigma\), where
 \(\sigma\) is a nonnegative measure on \(\mathbb{T}\).
 If
 \begin{equation}
 \label{TF}
 f(t)=\sum\limits_{1\leq k\leq n}\xi_k\,e(t,z_k),\quad
 g(t)=\sum\limits_{1\leq k\leq n}\eta_k\,e(t,z_k)
 \end{equation}
 belong to \(\mathscr{L}(z_1,z_2,\,\ldots\,,z_n)\), then their scalar product
 in the space \(L^2_\sigma\) is equal to
 \begin{equation}
 \label{ScP}
 \langle\,f\,,\,g\,\rangle_{\!_{L^2_\sigma}}=
 \sum\limits_{1\leq k,l\leq n}\xi_k\,\overline{\eta_l}\,\langle\,e(.,z_k)\,,%
 \,e(.,z_l)\,\rangle_{\!_{L^2_\sigma}}.
 \end{equation}
 \emph{Among all nonnegative measures \(\sigma\) on \(\mathbb{T}\), we distinguish  the normalised Lebesgue measure \(m(dt)\)}.

 We discuss the following\\
 \textbf{Problem 1.} \textit{Let pairwise different points \(z_1,z_2,\,\ldots\,,z_n\) from \(\mathbb{D}_+\) be given. How to describe those nonnegative measures \(\sigma\) on \(\mathbb{T}\) for which the property
 \begin{equation}
\label{CoSP}
\langle\,f\,,\,g\,\rangle_{\!_{L^2_\sigma}}
=\langle\,f\,,\,g\,\rangle_{\!_{L^2_m}},\quad \forall\,f,g\in
\mathscr{L}(z_1,z_2,\,\ldots\,,z_n).
\end{equation}
  holds}?

\begin{defn}
\label{DSM}
 Given the set  \(z_1,z_2,\,\ldots\,,z_n\) of pairwise different points from \(\mathbb{D}_+\), let
  \(\Sigma(z_1,z_2,\,\ldots\,,z_n)\) be
 the set of all non-negative measures on \(\mathbb{T}\) for which the property \eqref{CoSP} holds.
\end{defn}
\begin{lem}
\label{PrSi}
 Given the set  \(z_1,z_2,\,\ldots\,,z_n\) of pairwise different points from \(\mathbb{D}_+\), then the set \(\Sigma(z_1,z_2,\,\ldots\,,z_n)\) is
 not empty, convex, and bounded set of measures which is closed with
 respect to the weak topology.
 \begin{proof}
 The set \(\Sigma(z_1,z_2,\,\ldots\,,z_n)\) is
 not empty: the Lebesgue measure \(m\) belongs to this set.
 The convexity of this set and its weak closeness are evident.
 The \(\Sigma(z_1,z_2,\,\ldots\,,z_n)\) is bounded. In  Lemma \ref{EsMe} we establish the estimate
\begin{equation}
\label{Est}
\sigma(\mathbb{T})\leq\frac{1+B(0)}{1-B(0)},\quad
\forall\,\sigma\in\Sigma(z_1,z_2,\,\ldots\,,z_n).
\end{equation}
(We emphasize that the inequalities \eqref{PoB} holds.)
 \end{proof}

\end{lem}

In view of \eqref{ScP}, the equality \eqref{CoSP} is equivalent to the totality of equalities
\begin{equation}
\label{TEq}
\langle\,e(.,z_k)\,,\,e(.,z_l)\,\rangle_{\!_{L^2_\sigma}}=
\langle\,e(.,z_k)\,,\,e(.,z_l)\,\rangle_{\!_{L^2_m}},\quad
1\leq k,l\leq n.
\end{equation}
The expressions in the right hand side of \eqref{TEq} can be calculated explicitly:
\begin{equation}
\label{ECa}
\langle\,e(.,z_k)\,,\,e(.,z_l)\,\rangle_{\!_{L^2_m}}=\frac{1}{1-z_k\overline{z_l}},
\quad 1\leq k,l\leq n.
\end{equation}
The  expressions in the left hand side of \eqref{TEq} can be presented as
 \begin{equation}
\label{PrA}
\langle\,e(.,z_k)\,,\,e(.,z_l)\,\rangle_{\!_{L^2_\sigma}}=
\int\limits_{\mathbb{T}}\frac{1}{t-z_k}\frac{1}{\overline{t-z_l}}\,\sigma(dt),
\quad 1\leq k,l\leq n.
\end{equation}

Thus the Problem 1 can be reformulated as follows:\\
\textbf{Problem \(\boldsymbol{1^\prime}\)}.  \textit{Let pairwise different points \(z_1,z_2,\,\ldots\,,z_n\) from \(\mathbb{D}_+\) be given. How to describe such nonnegative measures \(\sigma\) on \(\mathbb{T}\) for which the totality of equalities
\begin{equation}
\label{PIC}
\int\limits_{\mathbb{T}}\frac{1}{t-z_k}\frac{1}{\overline{t-z_l}}\,\sigma(dt)=
\frac{1}{1-z_k\overline{z_l}},\quad 1\leq k,l\leq n,
\end{equation}
hold}?

It turns out that the Problem \(1^\prime\) can be reduced to a very special Newanlinna-Pick interpolation problem.

Let \(\sigma\) be a nonnegative measure on \(\mathbb{T}\). We associate
 the function
\begin{equation}
\label{AsF}
\varphi_{\sigma}(z)=\int\limits_{\mathbb{T}}\frac{t+z}{t-z}\,\sigma(dt), \quad z\in\mathbb{D_+},
\end{equation}
with this \(\sigma\).
It is clear that the function \(\varphi_{\sigma}\) is holomorphic
 in \(\mathbb{D}_+\).

If \(\zeta^{\prime},\zeta^{''}\in\mathbb{D}_+\), then
\begin{equation}
\label{Id}
\frac{\varphi_{\sigma}(\zeta^{\prime})+
\overline{\varphi_{\sigma}(\zeta^{''})}}{2(1-\zeta^{'}\overline{\zeta^{''}})}=
\int\limits_{\mathbb{T}}\frac{1}{t-\zeta^{'}}\,\frac{1}{\overline{t-\zeta^{''}}}\,\sigma(dt).
\end{equation}
For \(\zeta^{'}=z_k,\,\zeta^{''}=z_l\), the expression in the right hand %
side
of \eqref{Id} coincides with the expression in the left hand side of \eqref{PIC}. Thus the equalities \eqref{PIC} can be rewritten in the form
\begin{equation*}
\frac{\varphi_{\sigma}(z_k)+
\overline{\varphi_{\sigma}(z_l)}}{1-z_k\overline{z_l}}=
\frac{2}{1-z_k\overline{z_l}}, \quad 1\leq k,l\leq n,
\end{equation*}
that is
\begin{equation}
\label{OF}
\varphi_{\sigma}(z_k)+\overline{\varphi_{\sigma}(z_l)}=
2,\quad 1\leq k,l\leq n.
\end{equation}

Thus the following statement is proved.
\begin{lem}
\label{CoCo}
Let \(z_1,z_2,\,\ldots\,,z_n\) be pairwise different points of the unit disc
\(\mathbb{D}_+\).
For  a non-negative measure \(\sigma\) on \(\mathbb{T}\), let
\(\varphi_\sigma\) be the function  associated with the measure \(\sigma\) according to \eqref{AsF}.
\begin{enumerate}
\item
If \(\sigma\in\Sigma(z_1,z_2,\,\ldots\,,z_n)\),
then the equalities \eqref{OF} hold.
\item
If  the equalities \eqref{OF} hold, then
 the \(\sigma\in\Sigma(z_1,z_2,\,\ldots\,,z_n)\).
\end{enumerate}
\end{lem}
We consider the equalities \eqref{OF} as a system of equations with respect to the values \(\varphi_{\sigma}(z_k)\).
\begin{lem}
\label{SpSys}
Let \(\varphi_1,\varphi_2,\,\ldots\,,\varphi_n\) be complex numbers satisfying the system of equations
\begin{equation}
\label{soe}
\varphi_k+\overline{\varphi_l}=2,\quad 1\leq k,l\leq n.
\end{equation}
Then there exists a real number \(\beta\) such that
\begin{equation}
\label{sol}
\varphi_k=1-i\beta,\quad 1\leq k\leq n.
\end{equation}
Vice versa, if \(\beta\) is an arbitrary real number and \(\varphi_k\) are
defined by \eqref{sol}, then the equalities \eqref{soe} hold.

Thus the  set
of solutions of the system of equations \eqref{soe} is an one-parametric family. The arbitrary real number \(\beta\) is a free parameter which parameterizes this family according to \eqref{sol}.
\end{lem}
\begin{proof} Let \(\boldsymbol{e}=[1,1,\,\ldots\,,1]\) and
be  \(\boldsymbol{\varphi}=[\varphi_1,\varphi_2,\,\ldots\,,\varphi_n]\)
be \(1\times n\) matrices (rows), \(\boldsymbol{e}^\ast\) and
\(\boldsymbol{\varphi}^\ast\) be the Hermitian conjugate matrices (columns).
The system \eqref{soe} can be rewritten in a matrix form
\begin{equation}
\label{mf}
\boldsymbol{\varphi}^\ast\boldsymbol{e}+
\boldsymbol{e}^\ast\boldsymbol{\varphi}=2\,\boldsymbol{e}^\ast\boldsymbol{e}\,.
\end{equation}
The matrix in the right hand side of \eqref{mf} is of rank one. Thus \(\boldsymbol{\varphi}\) must be of the form
\(\boldsymbol{\varphi}=\zeta\,\boldsymbol{e}\), where \(\zeta\in\mathbb{C}\).
(Otherwise the matrix in the left hand side of \eqref{mf} is of rank two.)
Substituting \(\boldsymbol{\varphi}=\zeta\,\boldsymbol{e}\) into \eqref{mf},
we see that \(\boldsymbol{\varphi}=(1-i\beta)\,\boldsymbol{e}\),
where \(\beta\in\mathbb{R}\).
\end{proof}
From the equalities \eqref{OF} and Lemma \ref{SpSys} it follows that there exists a real number
\(\beta\) such that the equalities
\begin{equation}
\label{FE}
\varphi_{\sigma}(z_k)=1-i\beta,\quad 1\leq k \leq n.
\end{equation}
hold.

Given a non-negative measure \(\sigma\) on \(\mathbb{T}\) and a real number \(\beta\), we associate the function
 \begin{equation}
\label{cb}
c_{\sigma,\beta}(z)=i\beta+\int\limits_{\mathbb{T}}\frac{t+z}{t-z}\,\sigma(dt),
\quad z\in\mathbb{D}_+.
\end{equation}
with these \(\sigma\) and \(\beta\). It is clear that
\begin{equation}
\label{RePa}
\re c_{\sigma,\beta}(z)=h_{\sigma}(z),
\end{equation}
where
\begin{equation}
\label{AsFh}
h_{\sigma}(z)=\int\limits_{\mathbb{T}}\frac{1-|z|^2}{|t-z|^2}\,\sigma(dt), \quad z\in\mathbb{D_+}.
\end{equation}

The equalities \eqref{FE} can be rewritten as
\begin{equation}
\label{ReE}
c_{\sigma,\beta}(z_k)=1,\quad 1\leq k\leq n.
\end{equation}

Lemma \ref{CoCo} can be reformulated as follows.
\begin{lem}
\label{ReL}
Let \(z_1,z_2,\,\ldots\,,z_n\) be pairwise different numbers from \(\mathbb{D}_+\). For  a non-negative measure \(\sigma\) on \(\mathbb{T}\) and a real number \(\beta\), let
\(c_{\sigma,\beta}(z)\) be the function  associated with these \(\sigma\)
and \(\beta\) according to \eqref{AsF}.
\begin{enumerate}
\item
If \(\sigma\in\Sigma(z_1,z_2,\,\ldots\,,z_n)\), then there exists the real number
\(\beta\) such that the interpolation conditions \eqref{ReE} hold
for the function \(c_{\sigma,\beta}\).
\item
If the interpolation conditions \eqref{ReE} hold for some function \(c_{\sigma,\beta}\), then \(\sigma\in\Sigma(z_1,z_2,\,\ldots\,,z_n)\).
\end{enumerate}
\end{lem}
The functions of the form \eqref{ReE}, where \(\sigma\) is an arbitrary
non-negative measure on \(\mathbb{T}\) and \(\beta\) is an arbitrary real
number, can be characterized by their properties.
\begin{defn}
\label{DeCaCl} \emph{The Caratheodory class} \(\boldsymbol{\mathcal{C}}\) is the class of all functions \(c(z)\)
which possess the properties:
\begin{enumerate}
\item
The function \(c(.)\) is defined and holomorphic in the disc \(\mathbb{D}_+\).
\item
The real part of the function \(c(.)\) is non-negative in \(\mathbb{D}\):
\begin{equation}
\label{PRP}
\re \,c(z)\geq 0\quad \forall z\in\mathbb{D}_+.
\end{equation}
\end{enumerate}
\end{defn}

\textbf{The Representation Theorem for the Caratheodory class.}
\begin{enumerate}
\item
\textit{Let \(\sigma\) be a non-negative measure on \(\mathbb{T}\), \(\beta\) is
a real number, and \(c_{\sigma,\beta}(.)\) be the function defined by
the equality \eqref{cb}. Then the function \(c_{\sigma,\beta}(.)\)
belongs to the Caratheodory class \(\boldsymbol{\mathcal{C}}\).}
\item
\textit{Let \(c(.)\) be a function which belongs to the Caratheodory class
\(\boldsymbol{\mathcal{C}}\). Then the function \(c(.)\) is representable
in the form \eqref{cb}:
\begin{equation}
\label{rcc}
c(z)=c_{\sigma,\beta}(z)
\end{equation}
with some non-negative measure \(\sigma\) and real number \(\beta\).
The measure \(\sigma\) and number \(\beta\) are determined by the
function \(c(.)\) uniquely.}
\end{enumerate}

\vspace{2.0ex}
The next statement is an immediate consequence of
Lemma \ref{ReL} and of the Representation Theorem.
\begin{lem}
Let \(z_1,z_2,\,\ldots\,,z_n\) be pairwise different points from \(\mathbb{D}_+\).
Then the set of non-negative measures \(\sigma\)
belonging to the set \(\Sigma(z_1,z_2,\,\ldots\,,z_n)\)
coincides with the set of measures \(\sigma\) appearing as the representing
measures, \eqref{rcc}, for those functions \(c(.)\in\boldsymbol{\mathcal{C}}\) which satisfy the interpolating
conditions
\begin{equation}
\label{ICCa}
c(z_k)=1,\quad 1\leq k\leq n.
\end{equation}
\end{lem}

Thus the original problem, i.e. Problem 1, is related to the following
interpolation problem in the Caratheodory class. \\
\textbf{Problem \(\boldsymbol{2_{\,\mathcal{C}}}\).}
\textit{
 Let \(z_1,z_2,\,\ldots\,,z_n\) be pairwise different points from \(\mathbb{D}_+\).
 \emph{The function \(c\) is a solution of the
Problem~\(\boldsymbol{2_{\,\mathcal{C}}}\)} if \(c\in\boldsymbol{\mathcal{C}}\),
and  the interpolation conditions \eqref{ICCa} are satisfied.}

This interpolation
problem is the \emph{Nevanlinna-Pick problem in the class \(\mathcal{C}\)} with  generic interpolation nodes, but with  very special interpolation
values.

There is the very well developed machinery for study
the Nevanlinna-Pick problems in various functional classes. (See \cite{Kov}.)
This machinery includes the solvability criteria as well as
the description of the set of all solutions in the case of solvability.
We do not need to use such machinery. The set of solutions of the Nevanlinna-Pick problem \eqref{ICCa} in the class \(\boldsymbol{\mathcal{C}}\) can be easily describe without using the mentioned general theory.

\begin{defn}
\label{DeShCl} \emph{The Schur class} \(\boldsymbol{\mathcal{S}}\) is the class of all functions \(s(z)\)
which possess the properties:
\begin{enumerate}
\item
The function \(s(.)\) is defined and holomorphic in the disc \(\mathbb{D}_+\).
\item
The absolute value of the function \(s(.)\) does not exceed one in \(\mathbb{D}\):
\begin{equation}
\label{PRS}
|s(z)|\leq 1\quad \forall z\in\mathbb{D}_+.
\end{equation}
\end{enumerate}
\end{defn}

\vspace{3.0ex}
\noindent
The fractional linear transformation
\begin{equation}
\label{FrLT}
c(z)=\frac{1+s(z)}{1-s(z)}
\end{equation}
establishes one-to one correspondence  between classes \(\boldsymbol{\mathcal{C}}\) and \(\boldsymbol{\mathcal{S}}\).

It is clear that if the functions \(c\) and \(s\) are related by the transformation \eqref{FrLT}, then the interpolation conditions
\begin{equation}
\label{ICSh}
s(z_k)=0, \quad 1\leq k\leq n,
\end{equation}
for the function \(s\) correspond to the interpolation conditions \eqref{ICCa} to the function \(c\).

Thus the interpolation Problem \(2_{\mathcal{\,C}}\) in the class \(\boldsymbol{\mathcal{C}}\) is reduced to the following interpolation
problem in the class \(\boldsymbol{\mathcal{S}}\).\\
\textbf{Problem \(\boldsymbol{2_{\,\mathcal{S}}}\).}
\textit{
Let \(z_1,z_2,\,\ldots\,,z_n\) be pairwise different points from \(\mathbb{D}_+\).
 The function \(s\) is a solution of the
Problem~\(\boldsymbol{2_{\,\mathcal{S}}}\) if \(s\in\boldsymbol{\mathcal{S}}\),
and  the interpolation conditions \eqref{ICSh} are satisfied.}

\begin{lem}
\label{ReCaSh}
Let \(c\) and \(s\) be functions related by the fractional linear transformation \eqref{FrLT}. The function \(c\) is a solution of the
interpolation Problem \(2_{\,\mathcal{C}}\) if and only if the function \(s\)
is a solution of the
interpolation Problem \(2_{\,\mathcal{S}}\).
\end{lem}
\begin{lem}
\label{DSIPs}
The interpolation \emph{Problem \(2_{\,\boldsymbol{\mathcal{S}}}\)} is
solvable. The formula
\begin{equation}
\label{SIPs}
s(z)=B(z)\,\omega(z),
\end{equation}
 where \(B(z)\) is the Blaschke product \eqref{BP} constructed from the given interpolation
nodes \(z_k,\,1\leq k\leq n\), parameterizes  the set of
all solutions \(s\) of the \emph{Problem \(2_{\,\boldsymbol{\mathcal{S}}}\)}
by means of arbitrary functions \(\omega\in\boldsymbol{\mathcal{S}}\):
\begin{enumerate}
\item
If \(s\) is a solution of the \emph{Problem \(2_{\,\boldsymbol{\mathcal{S}}}\)}, then \(s\) is representable in the form \eqref{SIPs}, where \(\omega\in\boldsymbol{\mathcal{S}}\).
\item
If \(\omega\in\boldsymbol{\mathcal{S}}\) and \(s\) is related to \(\omega\)
by \eqref{SIPs}, then \(s\) is a solution of
 \emph{Problem~\(2_{\,\boldsymbol{\mathcal{S}}}\)}.
 \item
 The correspondence \eqref{SIPs} between parameters
 \(\omega\in\boldsymbol{\mathcal{S}}\) and solutions \(s\) of
 \emph{Problem~\(2_{\,\boldsymbol{\mathcal{S}}}\)} is one-to-one.
\end{enumerate}
\end{lem}

In view of the mentioned relation between Problems \(2_{\mathcal{S}}\) and \(2_{\mathcal{C}}\), the following result is established.
\begin{lem}
\label{DSIPc}
The interpolation \emph{Problem \(2_{\,\boldsymbol{\mathcal{C}}}\)} is
solvable. The formula
\begin{equation}
\label{SIPc}
c(z)=\frac{1+B(z)\,\omega(z)}{1-B(z)\,\omega(z)},
\end{equation}
 where \(B(z)\) is the Blaschke product \eqref{BP} constructed from the given interpolation
nodes \(z_k,\,1\leq k\leq n\), parameterizes  the set of
all solutions \(c\) of the \emph{Problem \(2_{\,\boldsymbol{\mathcal{C}}}\)}
by means of arbitrary functions \(\omega\in\boldsymbol{\mathcal{S}}\):
\begin{enumerate}
\item
If \(c\) is a solution of the \emph{Problem \(2_{\,\boldsymbol{\mathcal{C}}}\)}, then \(c\) is representable in the form \eqref{SIPc}, where \(\omega\in\boldsymbol{\mathcal{S}}\).
\item
If \(\omega\in\boldsymbol{\mathcal{S}}\) and \(c\) is related to \(\omega\)
by \eqref{SIPc}, then \(c\) is a solution of
 \emph{Problem~\(2_{\,\boldsymbol{\mathcal{C}}}\)}.
 \item
 The correspondence \eqref{SIPc} between parameters
 \(\omega\in\boldsymbol{\mathcal{S}}\) and solutions \(c\) of
 \emph{Problem~\(2_{\,\boldsymbol{\mathcal{C}}}\)} is one-to-one.
\end{enumerate}
\end{lem}

Lemmas \ref{ReL} and \ref{DSIPc} allow to describe the set of all
measures \(\sigma\) which belongs to the set \(\Sigma(z_1,z_2,\,\ldots\,,z_n)\). Holomorphic functions
\(c_{\sigma,\beta}\) from the Caratheodory class appear in such description.
However it is more convenient to formulate the final result not in terms
of functions of the Caratheodory class but in terms of real parts such functions. This allows to eliminate the value \(\beta\) which is not
related to the original problem.

\begin{defn}
\label{DeHaCl} \emph{The Herglotz class} \(\boldsymbol{\mathcal{H}}\) is the class of all functions \(h(z)\)
which possess the properties:
\begin{enumerate}
\item
The function \(h(.)\) is defined and harmonic in the disc \(\mathbb{D}_+\).
\item
The  function \(h(.)\) is non-negative in \(\mathbb{D}\):
\begin{equation}
\label{PRPh}
h(z)\geq 0,\quad \forall z\in\mathbb{D}_+.
\end{equation}
\end{enumerate}
\end{defn}

\vspace{3.0ex}
\noindent
The following theorem characterizes functions \(h\in\boldsymbol{\mathcal{H}}\) as
functions which admit the representation
 \begin{equation}
\label{RHh}
h(z)=\int\limits_{\mathbb{T}}\frac{1-|z|^2}{|t-z|^2}\,\sigma(dt), \quad z\in\mathbb{D_+},
\end{equation}
with a non-negative measure \(\sigma\).

\vspace{3.0ex}
\noindent
\textbf{Riesz-Herglotz Representation Theorem.}
\textit{
\begin{enumerate}
\item
Let \(h\) be a function from the class \(\boldsymbol{\mathcal{H}}\).
Then the function \(h\) is representable in the form \eqref{RHh}, where
\(\sigma\) is a non-negative measure on \(\mathbb{T}\).
\item
Let \(\sigma\) be a non-negative measure on \(\mathbb{T}\) and the function
\(h\) be constructed from this \(\sigma\) according to \eqref{RHh}.
Then \(h\in\boldsymbol{\mathcal{H}}\).
\item
The correspondence \eqref{RHh} between \(h\) and \(\sigma\) is one-to-one.
\end{enumerate}
}

\vspace{3.0ex}
The following Theorem is the main result of this note.
\begin{theo}
\textit{
Let \(z_1,z_2,\,\ldots\,,,z_n\) be pairwise different points from
\(\mathbb{D}_+\) and \(B(z)\) be the Blaschke product constructed
from these points, \eqref{BP}.\\
The formula
\begin{equation}
\label{SIPh}
h(z)=\re \frac{1+B(z)\,\omega(z)}{1-B(z)\,\omega(z)}\ccomma \quad z\in
\mathbb{D}_+,
 \end{equation}
parameterises the set \(\Sigma(z_1,z_2,\,\ldots\,,z_n)\).
The function \(\omega\in\boldsymbol{\mathcal{S}}\) serves as a parameter.}
\textit{
\begin{enumerate}
\item
Let \(\omega(z)\) be a function from the Schur class \(\boldsymbol{\mathcal{S}}\) and the function
\(h\) be constructed from this \(\omega\) by the formula \eqref{SIPh}.\\
Then
 \begin{enumerate}
 \item
 The function \(h\) is a harmonic function from the Herglotz class
 \(\boldsymbol{\mathcal{H}}\).
 \item
 The measure \(\sigma\) which appears in the Riesz-Herglotz representation \eqref{RHh}
 of the function \(h\) belongs to the set \(\Sigma(z_1,z_2,\,\ldots\,,z_n)\).
 \end{enumerate}
 \item
 Let \(\sigma\) be a measure which belongs to the set \(\Sigma(z_1,z_2,\,\ldots\,,z_n)\). Let \(h\) be
 the function which is constructed from this \(\sigma\) according to
 \eqref{RHh}.\\
 Then there exists the function \(\omega(z)\) from the
 Schur class \(\boldsymbol{\mathcal{S}}\) such that the equality
 \eqref{SIPh} holds.
 \item
 Different functions \(h\) correspond to different parameters \(\omega\).
\end{enumerate}
}
\end{theo}
 The parametrization \eqref{SIPh} allows to estimate the values
 \(\sigma(\mathbb{T})\) for the measures \(\sigma\in\Sigma(z_1,z_2,\,\ldots\,,z_n)\).
\begin{lem}
\label{EsMe}
Let \(z_1,z_2,\,\ldots\,,z_n\) be pairwise different points from \(\mathbb{D}_+\). Then
\begin{subequations}
\label{esti}
\begin{align}
\max\{\sigma(\mathbb{T}):\,\sigma\in\Sigma(z_1,z_2,\,\ldots\,,z_n)\}&=
\frac{1+B(0)}{1-B(0)}\ccomma\label{Esti+}\\
\min\{\sigma(\mathbb{T}):\,\sigma\in\Sigma(z_1,z_2,\,\ldots\,,z_n)\}&=
\frac{1-B(0)}{1+B(0)}\cdot\label{Esti=}
\end{align}
\end{subequations}
\end{lem}
\begin{proof}
Let \(\sigma\) be the measure which represents the function \(h\) from \eqref{SIPh}. Then
\begin{equation}
\label{EV}
\sigma(\mathbb{T})=h(0)=\re \frac{1+B(0)\omega(0)}{1-B(0)\omega(0)}\cdot
\end{equation}
 Since
\(B(0)=\prod\limits_{1\leq k\leq n}|z_k|,\) the inequalities
\begin{equation}
\label{PoB}
0\leq B(0)<1
\end{equation}
 hold.
  When the parameter \(\omega\) runs over the whole Schur class
\(\boldsymbol{\mathcal{S}}\), the value \(\omega(0)\) runs over the \emph{closed} unit disc
\(\overline{\mathbb{D}}_+=\mathbb{D}_+\cup\mathbb{T}\).
Therefore the value in the right hand side of \eqref{EV} attained its
maximum for \(\omega(0)=+1\) and its minimum for \(\omega(0)=-1\).
The maximum corresponds to the choice of the parameters \(\omega(z)\equiv+1\), the minimum corresponds to the choice of the parameters \(\omega(z)\equiv~-1\).
\end{proof}

\vspace{1.0ex}
\noindent
\textbf{Remark.} The space \(\mathscr{L}(z_1,z_2,\,\ldots\,,z_n)\) can be identified with the subspace \(H^2_-\ominus B^{-1}H^2_-\) of the Hilbert space \(L^2_m\), where
\(H^2_-\) is the Hardy class in the domain \(\mathbb{D}_-\) and \(\ominus\)
is the orthogonal complement with respect to the scalar product in \(L^2_m\).


\begin{thebibliography}{1111}
\bibitem[Kov]{Kov} Kovalishina, I.V. \textit{Analytic Theory of a class of interpolation problem.}\\
     Mathematics of the USSR-Izvestiya, 1984,
    \textbf{22}:3, 419-463.
\end{thebibliography}
\end{document}